\documentclass[a4paper,leqno]{mathscan}

\usepackage[all]{xy}
\usepackage{graphicx}
 \usepackage{verbatim}
\usepackage[danish,english]{babel}
\usepackage{color}
\usepackage{enumerate}
\usepackage[abs]{overpic}

\newcommand{\R}{\mathbb R}
\newcommand{\N}{\mathbb N}
\newcommand{\Z}{\mathbb Z}

\newcommand{\union}{\cup}
\newcommand{\intersection}{\cap}

\newcommand{\cisom}[1]{[#1]_{\tiny{\mathrm{isom}}}}
\newcommand{\Laff}{L_{\text{aff}}}
\newcommand{\vx}{\underline x}
\newcommand{\vy}{\underline y}

\newtheorem{thm}{Theorem}[section]
\newtheorem{pro}[thm]{Proposition}
\newtheorem{cor}[thm]{Corollary}
\newtheorem{lem}[thm]{Lemma}

\theoremstyle{definition}

\newtheorem{rem}[thm]{Remark}       
\newtheorem{defn}[thm]{Definition}  
\newtheorem{exam}[thm]{Example}     

\usepackage[hidelinks]{hyperref}
\begin{document}

\title[Inverse limit for a combinatorial tiling]
 {Continuous hull of a combinatorial pentagonal tiling as an inverse limit}

\author{Maria Ramirez-Solano}
\thanks{}

\address{Department of Mathematics, University of Copenhagen, Universitetsparken 5,
2100 K\o benhavn \O ,
 Denmark.}

\email{mrs@math.ku.dk}



\thanks{Supported by the Danish National Research Foundation through the Centre
for Symmetry and Deformation (DNRF92), and by the Faculty of Science of the University of Copenhagen.}

\begin{abstract}
In \cite{MRScontinuoushull} we constructed a compact topological space for the combinatorics of "A regular pentagonal tiling of the plane", which we call the continuous hull. We also constructed a substitution map on the space which turns out to be a homeomorphism, and so the pair given by the continuous hull and the substitution map yields a dynamical system. In this paper we show how we can write this dynamical system as another  dynamical system given by an inverse limit and a right shift map.
\end{abstract}

\maketitle

For an aperiodic FLC Euclidean substitution tiling of the plane, there is a recipe for writing its continuous hull as an inverse limit. See for instance \cite{PutnamBible95}, \cite{Sadun08}. Such recipe consists of the following two steps:
(1) If necessary, write the substitution map as one that "forces its border" with a so-called collared substitution.
(2) Construct an equivalence relation on the continuous hull so that the continuous hull modulo this equivalence relation is a finite CW-complex, which is also known as the Anderson-Putnam finite CW-complex.
The inverse system consists of repeatedly using the same object, namely the  Anderson-Putnam finite CW-complex, and a morphism, namely a continuous surjective map induced by the quotient and substitution map.\\

In \cite{MRScontinuoushull} we defined the continuous hull  $\Omega$ for the combinatorics $K$ of "A regular pentagonal tiling of the plane" as
\begin{center}
    \begin{tabular}{   @{}r@{}  @{}p{10cm}@{} }
   $\Omega:=\{\cisom{L,x}\mid $ & $\, L$  is a combinatorial tiling locally isomorphic to $K$ and point  $x\in \Laff\},$
    \end{tabular}
\end{center}
where the isomorphism class is cell-preserving, decoration preserving,   isometric on each cell, and preserves the origin $x$.
We also equipped it with the following metric:
\begin{defn}[metric $d$ on $\Omega$]\index{metric $d$ on $\Omega$}
Define $d:\Omega\times\Omega\to \R_+$ by
$$d(\cisom{L,x},\cisom{L',x'}):=\min(\tfrac{1}{\sqrt2},\inf\Lambda),$$
where  $\Lambda\subset \R_+$, and $\varepsilon\in\Lambda$  if there exists maps
\begin{eqnarray*}
&& \phi:B(x,1/\varepsilon,L)\to L',\,\,\,\qquad  d_{L'}(\phi(x),x')\le \varepsilon\\
&& \phi':B(x',1/\varepsilon,L')\to L,\qquad d_{L}(\phi'(x'),x)\le \varepsilon
\end{eqnarray*}
 which are  cell-preserving maps, are isometries, preserve the decorations and degree of the vertices.
\end{defn}
We also defined a substitution map $\omega:\Omega\to\Omega$ via the decorated subdivison rule shown in Figure \ref{f:subdivisionmapnewdecorated}.
In the same article \cite{MRScontinuoushull} we showed that the metric space $(\Omega,d)$ is compact and that $\omega$ is a homeomorphism.\\

We follow the above two steps as guidance to the construction of the inverse limit.
The first difficulty is to find all the "collared tiles" for the collared substitution. We use the first section to address this problem, and we find  that there are 36 collared tiles, 45 collared edges, and 10 collared vertices (Theorem \ref{t:collaredtilesedgesvertices}).
The second difficulty is to define the equivalence relation to construct the finite CW-complex, and to show that everything is as it should be.
We use the the last two sections on this. The second main result is Theorem \ref{t:topologicalconjugateOmegaOmega1}, whose proof is adapted from \cite{PutnamBible95}, that shows that the dynamical system ($\Omega$,$\omega$) can be written as a dynamical system given by an inverse limit and a right shift map.
\begin{figure}
  \centering
  \includegraphics[scale=1]{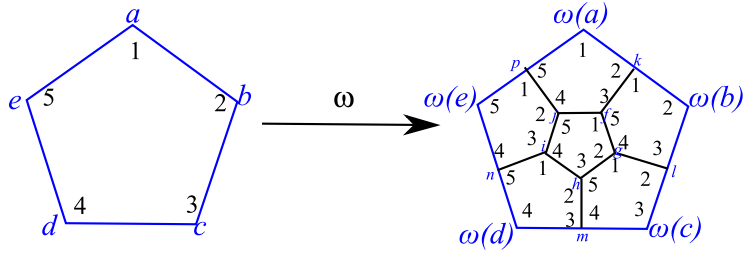}\\
  \caption{Subdivision map. }\label{f:subdivisionmapnewdecorated}
\end{figure}

\section{Collared substitution $\omega$ on $\Omega$}\label{s:collaredsubstitution}
In \cite{MRScontinuoushull} we constructed a substitution map $\omega:\Omega\to\Omega$ via the subdivision rule shown in Figure \ref{f:subdivisionmapnewdecorated}. Ignoring decoration and degree of the vertices of each pentagon, we have only one prototile, namely a pentagon.
If we ignore decoration, but include the degree of the vertices in the definition of prototiles, then we have 3 prototiles (See Figure 7 in \cite{MRSdiscretehull}).
If we include decoration (inside the pentagon) and degree of vertices then we have 11 prototiles (See Figure 22 in \cite{MRSdiscretehull}).
Unfortunately this substitution map does not "force its border", a term which will be defined below, and the fix is to construct a "collared substitution", as explained below. We borrow these two terms from \cite{Sadun08}. We remark that if we include the decoration (inside and outside of the pentagon) and degree of vertices in the definition of prototiles then we would have 36 prototiles according to Theorem \ref{t:collaredtilesedgesvertices}, something that would have spared us from introducing the term "collared substitution" but nothing more.

\begin{defn}[forcing its border]
  A substitution $\omega$ with prototiles $t_1,\ldots,$ $t_n$ is  said to \emph{force its border} if there is a $k\in\N$ such that
  every two  level-$k$ supertiles $\omega^k(t)$, $\omega^k(t')$ of same type (i.e. $t$ and $t'$ are copies of some prototile $t_i$) have same pattern of neighboring tiles.
\end{defn}

Our substitution map $\omega$ does not force its border because for instance, the two supertiles $w^2(t)$, $w^2(t')$  (of type $t_{2,5}$ as seen in Figure 22 in \cite{MRSdiscretehull}) shown in green and blue colors in Figure \ref{f:wdoesnotforceitsborder}
have a different pattern of neighbors as they do not agree on the yellow ribbon.  No matter how many times we substitute $t$ and $t'$ we will never have them to agree for the decoration on the yellow ribbons replicate themselves as suggested in Figure \ref{f:Knfb}.
\begin{figure}[htbp]
  \begin{minipage}[b]{0.5\linewidth}
    \centering
    \begin{overpic}[width=\linewidth]{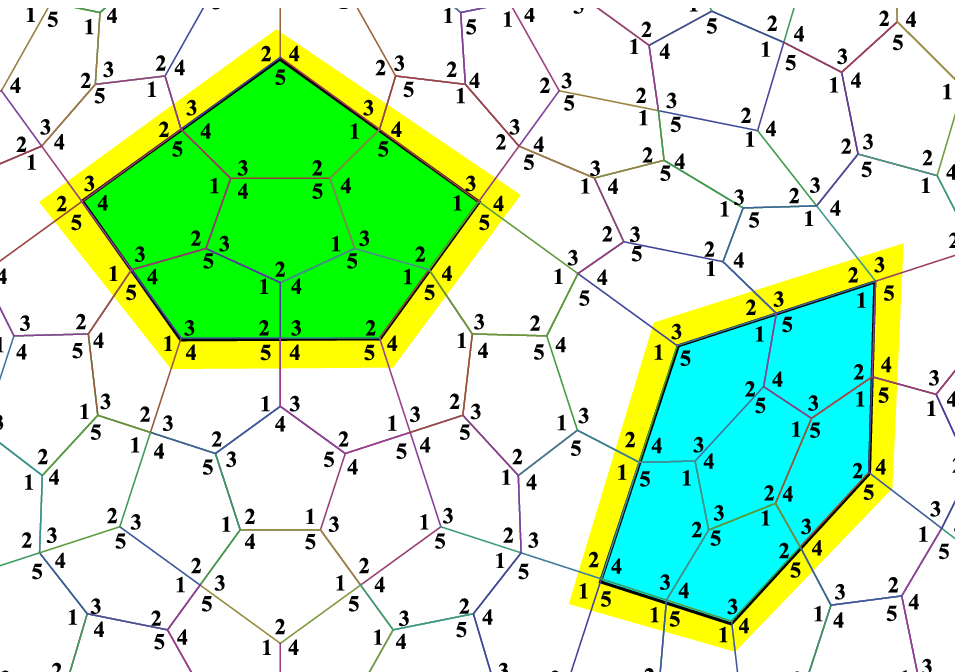}
    \put(45,84){\tiny green}
    \put(135,40){\tiny blue}
    \put(80,5){\tiny yellow ribbon}
     \put(30,52){\tiny yellow ribbon}
    \end{overpic}
    \caption{Two level-2-supertiles of type $t_{2,5}$ have different neighbors.}
    \label{f:wdoesnotforceitsborder}
  \end{minipage}
  \begin{minipage}[b]{0.48\linewidth}
    \centering
    \begin{overpic}[width=\linewidth]{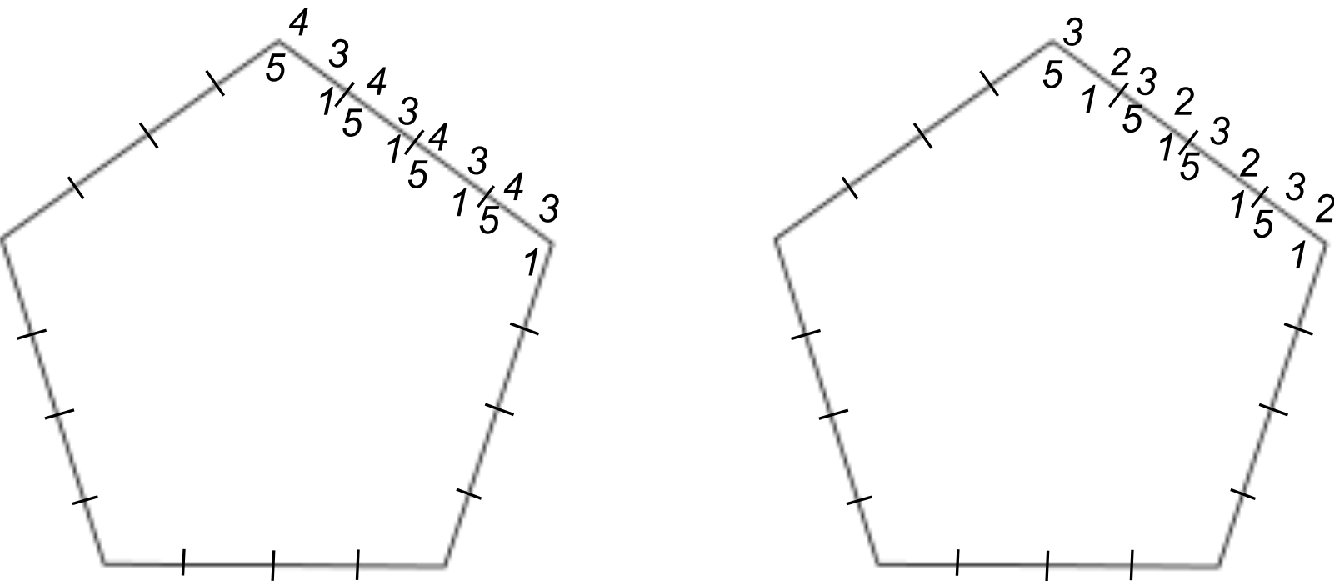}
    \end{overpic}
    \caption{The decoration of the yellow ribbons in the figure to the left duplicate themselves when subdivided.}
    \label{f:Knfb}
  \end{minipage}
\end{figure}

A natural question arises. Given a tiling generated by a substitution rule which does not force its border, can we find an alternative substitution which forces its border? In other words, is the property of forcing its border a property of the tiling or of the substitution itself?
This question was answered by  Anderson and Putnam in \cite{PutnamBible95}. It is a property of the substitution and not of the tiling.
In fact, they introduced a method, known as the Anderson-Putnam trick, to modify any substitution to one that forces its border.
The modified substitution consists in rewriting the substitution in terms of so-called collared tiles.
\begin{defn}[collared tiles]
  If we can label the tiles of a tiling  not only by their own type but by the pattern of their nearest neighbors, then we call such labels \emph{collared tiles}.
\end{defn}

We give an example of a collared substitution.
 \begin{exam}
 Suppose that $T=\ldots abbabbababbab\ldots$ is the Fibonacci tiling generated by the substitution rule $\sigma(a)=b$, $\sigma(b)=ab$. The neighbors of $a$, which we write in parenthesis, are always $a_1:=(b)a(b)$. The neighbors of $b$ are $b_1:=(a)b(b)$, $b_2:=(b)b(a)$, and $b_3:=(a)b(a)$, and  the patch $bbb$ never appears in the tiling.
The collared tiles are $a_1,b_1,b_2,b_3$. We remark that $a_1,b_1,b_2,b_3$ are not patches of three tiles, but single tiles that remember their neighbors.
The tiles $b_1,b_2,b_3$ as uncollared tiles are all the same as $b$, but as collared tiles they are three distinct copies of $b$.
The substitution forces its border because for  $k=2$ we have
\begin{eqnarray*}
&&\sigma^2(a_1)=(b_3)a_1b_1(b_2)\\
&&\sigma^2(b_1)=(b_1)b_2a_1b_1(b_2)\\
&&\sigma^2(b_2)=(b_1)b_2a_1b_3(a_1)\\
&&\sigma^2(b_3)=(b_1)b_2 a_1b_3(a_1).
\end{eqnarray*}
 \end{exam}
The more collared tiles we have, the more complicated the task of finding the integer $k$.

\begin{figure}[htbp]
  \begin{minipage}[b]{0.3\linewidth}
    \centering
    \includegraphics[width=\linewidth]{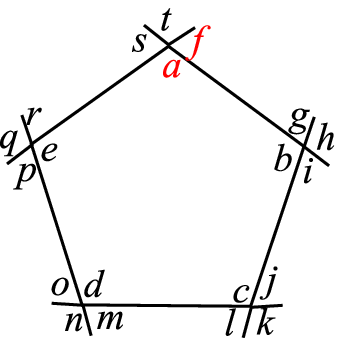}
    \caption{A collared tile.}
    \label{f:collaredtilenotation}
  \end{minipage}
  \hspace{0.5cm}
  \begin{minipage}[b]{0.5\linewidth}
    \centering
    \includegraphics[width=\linewidth]{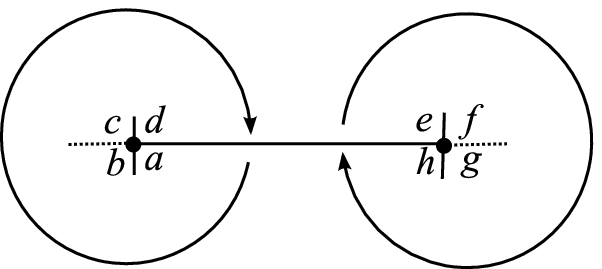}
    \caption{A collared edge and two collared vertices.}
    \label{f:collarededgetilenotation}
  \end{minipage}
\end{figure}

Back to our substitution rule, we need to determine the collared tiles.
In principle, we should look at all the neighbor pentagons of a pentagon to obtain the collared tiles.
But we do not need to do this because each neighbor tile is determined completely once we know the decorations of an edge,
and thus we are tailoring the Anderson-Putnam's trick to our needs.
We need to find all the possible decorations on the outside of each pentagon.
Any collared tile is of the form shown in Figure \ref{f:collaredtilenotation}. We will write it as the pair
$$((a,b,c,d,e), (f,g,h\mid i,j,k\mid l, m,n\mid o,p ,q\mid,r,s,t)),$$
 where the letters can take integer values from  1 to 5; the first ordered set $(a,b,c,d,e)$ represents the interior decoration, and the second one the exterior decoration.  Notice that $a$ and $f$ are decorations of both an edge and a vertex and the rest of the letters are written in clockwise order as shown in Figure  \ref{f:collaredtilenotation}. The letters $h,k,n,q,t$ are zero if the corresponding vertices are of degree $3$.
 We will further shorten our notation by writing the collared tile
 $$((1,2,3,4,5), (f,g,h\mid i,j,k\mid l, m,n\mid o,p ,q\mid,r,s,t))$$
 as
 $$(f,g,h\mid i,j,k\mid l, m,n\mid o,p ,q\mid,r,s,t).$$

 Any collared edge and vertex is of the form shown in Figure \ref{f:collarededgetilenotation}.
 We will write a collared edge and a collared vertex as
 $$(a,b,c,d\mid e,f,g,h),\qquad (a,b,c,d),$$
 respectively, where the letters take integer values from 1 to 5.
 The letters $b,c,f,g$ can take the value zero if the corresponding vertices are of degree $3$.

\begin{figure}
  \centering
  \includegraphics[scale=.7]{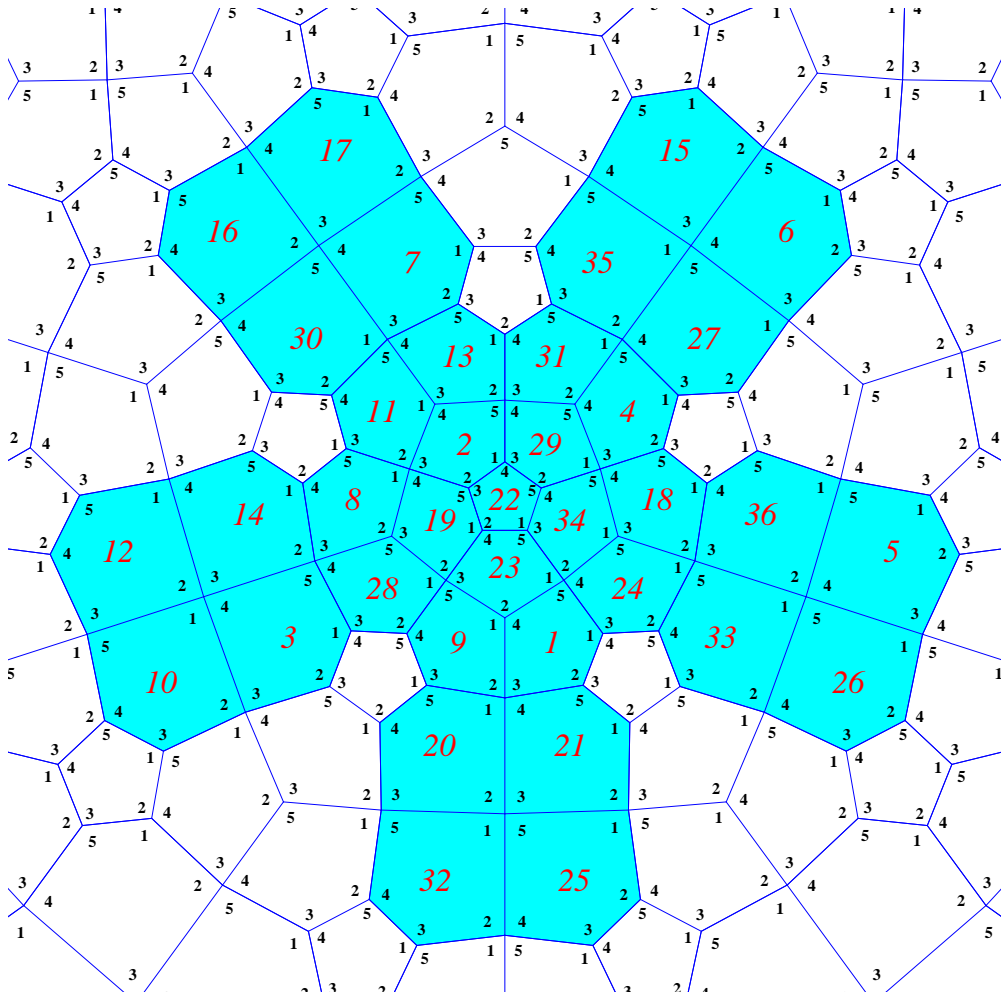}\\
  \caption{The collared tiles of $K$.}\label{f:collaredtiles}
\end{figure}

\begin{figure}
  \centering
  \begin{overpic}[scale=.7]{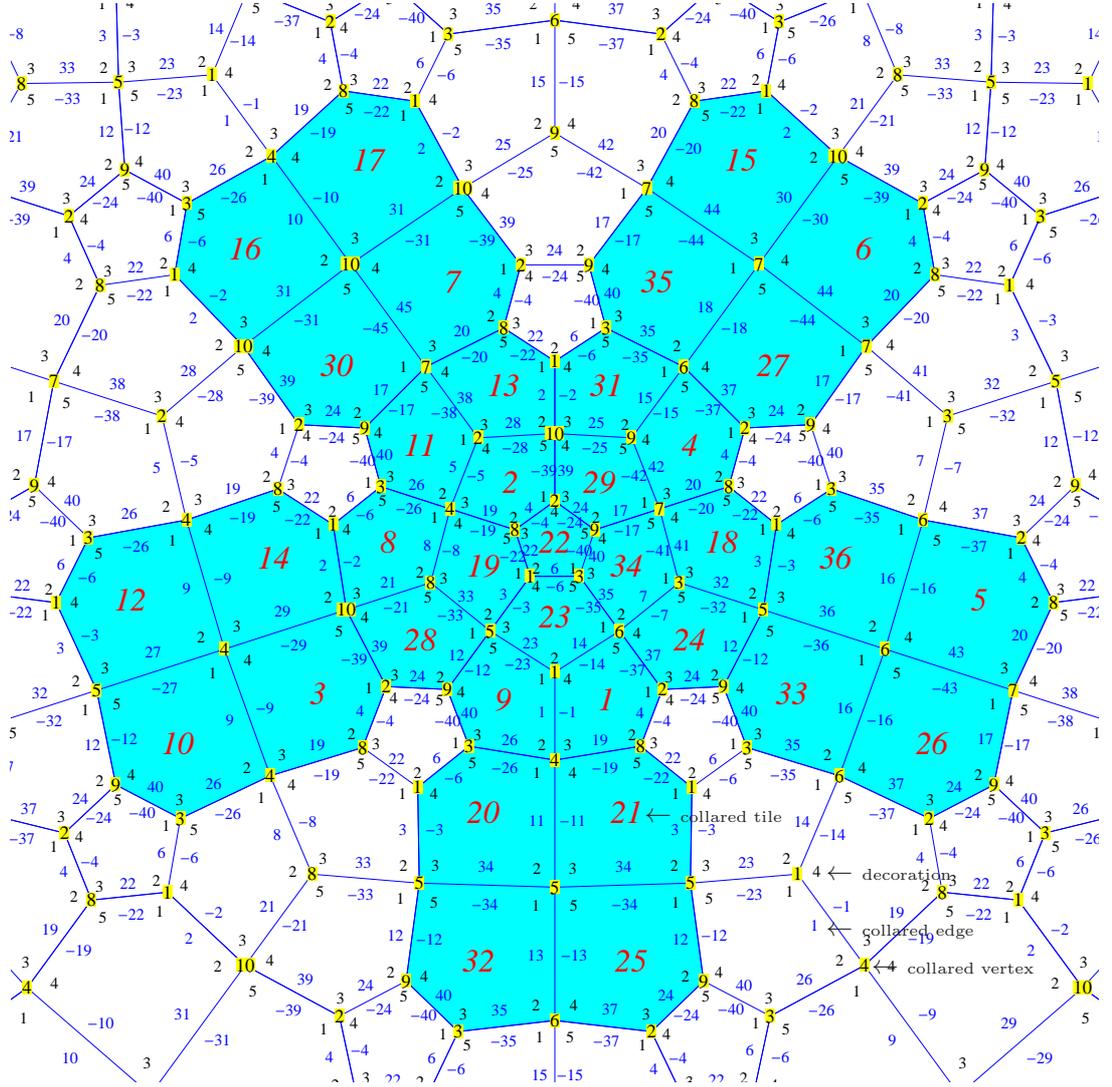}
  \put(237,98){$\leftarrow$ \tiny  collared tile}
  \put(305,55){$\leftarrow$ \tiny  collared edge}
  \put(322,41){$\leftarrow$ \tiny  collared vertex}
  \put(305,76){$\leftarrow$ \tiny  decoration}
  \end{overpic}
  \caption{The red numbers are the collared tiles. The blue numbers are the collared edges; we put a minus sign if it goes in reverse direction. The black numbers with yellow background are the collared vertices. The remaining black numbers are decorations.}\label{f:collaredtiles1}
\end{figure}

\begin{thm}\label{t:collaredtilesedgesvertices}
  The collared tiles, collared edges and collared vertices of $K$ are the ones shown in Table \ref{t:collaredtiles}, Table \ref{t:collarededges}, Table \ref{t:collaredvertices}, respectively. See Figure \ref{f:collaredtiles} and Figure \ref{f:collaredtiles1}.
\end{thm}

 \begin{table}
  \centering
\begin{tabular}{ |c | c |}
\hline
 & Collared tiles\\
\hline
 1 & $(4,3,0 \mid 5,4,1 \mid 2,1,0 \mid 2,1,2 \mid 4,3,0)$\\
 2 & $(4,3,0 \mid 5,4,1 \mid 2,1,0 \mid 3,2,3 \mid 4,3,0)$\\
 3 & $(4,3,0 \mid 5,4,1 \mid 2,1,2 \mid 3,2,3 \mid 4,3,0)$\\
 4 & $(4,3,0 \mid 5,4,5 \mid 1,5,0 \mid 2,1,2 \mid 4,3,0)$\\
 5 & $(4,3,0 \mid 5,4,5 \mid 1,5,1 \mid 2,1,2 \mid 4,3,0)$\\
 6 & $(4,3,0 \mid 5,4,5 \mid 1,5,1 \mid 3,2,3 \mid 4,3,0)$\\
 7 & $(4,3,0 \mid 5,4,5 \mid 1,5,2 \mid 3,2,3 \mid 4,3,0)$\\
 8 & $(4,3,0 \mid 5,4,5 \mid 2,1,0 \mid 2,1,0 \mid 3,2,3)$\\
 9 & $(4,3,4 \mid 1,5,0 \mid 1,5,0 \mid 2,1,2 \mid 3,2,0)$\\
10 & $(4,3,4 \mid 1,5,0 \mid 1,5,0 \mid 2,1,2 \mid 3,2,3)$\\
11 & $(4,3,4 \mid 1,5,0 \mid 1,5,0 \mid 2,1,3 \mid 4,3,0)$\\
12 & $(4,3,4 \mid 1,5,1 \mid 2,1,0 \mid 2,1,0 \mid 3,2,3)$\\
13 & $(4,3,4 \mid 5,4,0 \mid 1,5,1 \mid 3,2,0 \mid 3,2,0)$\\
14 & $(4,3,4 \mid 5,4,1 \mid 2,1,2 \mid 3,2,0 \mid 3,2,0)$\\
15 & $(4,3,4 \mid 5,4,5 \mid 1,5,1 \mid 3,2,0 \mid 3,2,0)$\\
16 & $(4,3,4 \mid 5,4,5 \mid 2,1,0 \mid 2,1,0 \mid 3,2,3)$\\
17 & $(4,3,4 \mid 5,4,5 \mid 2,1,2 \mid 3,2,0 \mid 3,2,0)$\\
18 & $(4,3,5 \mid 1,5,0 \mid 1,5,1 \mid 3,2,0 \mid 3,2,0)$\\
19 & $(4,3,5 \mid 1,5,0 \mid 2,1,2 \mid 3,2,0 \mid 3,2,0)$\\
20 & $(4,3,5 \mid 1,5,1 \mid 2,1,0 \mid 2,1,0 \mid 3,2,3)$\\
21 & $(4,3,5 \mid 1,5,1 \mid 2,1,2 \mid 3,2,0 \mid 3,2,0)$\\
22 & $(5,4,0 \mid 1,5,0 \mid 2,1,0 \mid 3,2,0 \mid 4,3,0)$\\
23 & $(5,4,0 \mid 1,5,1 \mid 2,1,0 \mid 2,1,0 \mid 3,2,4)$\\
24 & $(5,4,0 \mid 5,4,0 \mid 1,5,1 \mid 2,1,0 \mid 3,2,3)$\\
25 & $(5,4,0 \mid 5,4,0 \mid 1,5,1 \mid 2,1,2 \mid 3,2,3)$\\
26 & $(5,4,0 \mid 5,4,0 \mid 1,5,1 \mid 2,1,2 \mid 4,3,4)$\\
27 & $(5,4,0 \mid 5,4,0 \mid 1,5,1 \mid 2,1,3 \mid 4,3,4)$\\
28 & $(5,4,0 \mid 5,4,0 \mid 1,5,2 \mid 3,2,0 \mid 3,2,3)$\\
29 & $(5,4,0 \mid 5,4,0 \mid 1,5,2 \mid 3,2,0 \mid 4,3,4)$\\
30 & $(5,4,0 \mid 5,4,0 \mid 1,5,2 \mid 3,2,3 \mid 4,3,4)$\\
31 & $(5,4,0 \mid 5,4,5 \mid 2,1,0 \mid 2,1,0 \mid 3,2,4)$\\
32 & $(5,4,5 \mid 1,5,0 \mid 1,5,0 \mid 2,1,2 \mid 3,2,3)$\\
33 & $(5,4,5 \mid 1,5,0 \mid 1,5,0 \mid 2,1,2 \mid 3,2,4)$\\
34 & $(5,4,5 \mid 1,5,0 \mid 1,5,0 \mid 2,1,3 \mid 4,3,0)$\\
35 & $(5,4,5 \mid 1,5,0 \mid 1,5,0 \mid 2,1,3 \mid 4,3,4)$\\
36 & $(5,4,5 \mid 1,5,1 \mid 2,1,0 \mid 2,1,0 \mid 3,2,4)$\\
\hline
\end{tabular}
  \caption{Collared tiles of $K$.}\label{t:collaredtiles}
\end{table}

\begin{table}
  \centering
\begin{tabular}{ |c | c |c|}
\hline
& Collared edges & \\
\hline
 1 $(1,0,2,4 \mid 3,4,1,2)$ &  16 $(1,2,4,5 \mid 4,5,1,2)$ & 31 $(2,3,4,5 \mid 4,5,2,3)$\\
 2 $(1,0,2,4 \mid 3,4,5,2)$ &  17 $(1,3,4,5 \mid 4,0,5,2)$ & 32 $(2,3,5,1 \mid 5,0,1,3)$\\
 3 $(1,0,2,4 \mid 3,5,1,2)$ &  18 $(1,3,4,5 \mid 4,5,1,2)$ & 33 $(2,3,5,1 \mid 5,0,2,3)$\\
 4 $(1,0,3,4 \mid 3,0,5,2)$ &  19 $(2,0,3,5 \mid 4,1,2,3)$ & 34 $(2,3,5,1 \mid 5,1,2,3)$\\
 5 $(1,0,3,4 \mid 3,4,1,2)$ &  20 $(2,0,3,5 \mid 4,5,1,3)$ & 35 $(2,4,5,1 \mid 5,0,1,3)$\\
 6 $(1,0,3,5 \mid 4,0,1,2)$ &  21 $(2,0,3,5 \mid 4,5,2,3)$ & 36 $(2,4,5,1 \mid 5,1,2,3)$\\
 7 $(1,0,3,5 \mid 4,5,1,2)$ &  22 $(2,0,4,1 \mid 5,0,2,3)$ & 37 $(3,0,4,1 \mid 5,1,2,4)$\\
 8 $(1,2,3,4 \mid 3,0,5,2)$ &  23 $(2,0,4,1 \mid 5,1,2,3)$ & 38 $(3,0,4,1 \mid 5,1,3,4)$\\
 9 $(1,2,3,4 \mid 3,4,1,2)$ &  24 $(2,0,4,5 \mid 4,0,1,3)$ & 39 $(3,0,4,1 \mid 5,2,3,4)$\\
10 $(1,2,3,4 \mid 3,4,5,2)$ &  25 $(2,0,4,5 \mid 4,5,2,3)$ & 40 $(3,0,5,1 \mid 5,0,2,4)$\\
11 $(1,2,3,4 \mid 3,5,1,2)$ &  26 $(2,3,4,1 \mid 5,0,1,3)$ & 41 $(3,0,5,1 \mid 5,1,3,4)$\\
12 $(1,2,3,5 \mid 4,0,5,2)$ &  27 $(2,3,4,1 \mid 5,1,2,3)$ & 42 $(3,4,5,1 \mid 5,0,2,4)$\\
13 $(1,2,3,5 \mid 4,5,1,2)$ &  28 $(2,3,4,5 \mid 4,0,1,3)$ & 43 $(3,4,5,1 \mid 5,1,2,4)$\\
14 $(1,2,4,5 \mid 4,0,1,2)$ &  29 $(2,3,4,5 \mid 4,1,2,3)$ & 44 $(3,4,5,1 \mid 5,1,3,4)$\\
15 $(1,2,4,5 \mid 4,0,5,2)$ &  30 $(2,3,4,5 \mid 4,5,1,3)$ & 45 $(3,4,5,1 \mid 5,2,3,4)$\\
\hline
\end{tabular}
  \caption{Collared edges of $K$.}\label{t:collarededges}
\end{table}

\begin{table}
\centering
\begin{tabular}{|c|c|}
\hline
&Collared vertices\\
\hline
  1& $(1,0,2,4)$\\
2 &$(1,0,3,4)$\\
3 &$(1,0,3,5)$\\
4 &$(1,2,3,4)$\\
5 &$(1,2,3,5)$\\
6 &$(1,2,4,5)$\\
7 &$(1,3,4,5)$\\
8 &$(2,0,3,5)$\\
9 &$(2,0,4,5)$\\
10& $(2,3,4,5)$\\
\hline
\end{tabular}
  \caption{Collared vertices of $K$.}\label{t:collaredvertices}
\end{table}

\begin{figure}
  \centering
  \includegraphics[scale=1.5]{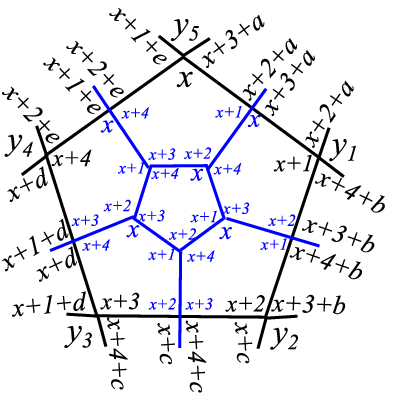}\\
  \caption{Possible collared tiles of the subdivision of a collared tile.}\label{f:collaring}
\end{figure}
\begin{proof}
We start by computing the collared tiles.
By Lemma 1.27 in \cite{MRSdiscretehull}, the possible collared tiles of the subdivision of a collared tile is shown in Figure \ref{f:collaring},  where
\begin{eqnarray*}
&&x=1,\ldots,5,\\
&&a,b,c,d,e,=0,1,
\end{eqnarray*}
and the symbols $y_1=y_2=y_3=y_4=y_5=0$ can take the zero value or
\begin{eqnarray*}
x+2+a&<y_1<&x+4+b\\
x+3+b&<y_2<&x+c\\
x+4+c&<y_3<&x+1+d\\
x+d&<y_4<&x+2+e\\
x+1+e&<y_5<&x+3+a.
\end{eqnarray*}

Reading the two possible patterns  from Figure \ref{f:collaring} we get

 \begin{eqnarray*}
  (1)&&((x,x+1,x+2,x+3,x+4),\\
  &&(x+4,x+3,0\mid x,x+4,0\mid x+1,x,0\mid x+2,x+1,0\mid x+3,x+2,0)).\\
  (2)&&((x,x+1,x+2,x+3,x+4),\\
  &&(x+3+a,x+2+a,x+3+a\mid x,x+4,0\mid x,x+4,0\mid\\
  && x+1,x,x+1+e\mid x+2+e,x+1+e,y_5)).
\end{eqnarray*}
 Running pattern (2) for all possible values of $x,a,e,$ and reorganizing them so the interior decoration is (1,2,3,4,5) we get
the possible exterior patterns\\

\begin{tabular}{ |c | c |c| c |}
\hline
  &   Exterior decoration of a tile & $y_5=0$ or & $i.e.\, y_5=$\\
  \hline
  1& $(4,3,4 \mid 1,5,0 \mid 1,5,0 \mid 2,1,2 \mid 3,2,y_5)$& $2 < y_5 < 4$ & 0,3\\
  2& $(4,3,4 \mid 1,5,0 \mid 1,5,0 \mid 2,1,3 \mid 4,3,y_5)$& $3 < y_5 < 4$ & 0\\
  3& $(5,4,5 \mid 1,5,0 \mid 1,5,0 \mid 2,1,2 \mid 3,2,y_5)$& $2 < y_5 < 5$ & 0,3,4\\
  4& $(5,4,5 \mid 1,5,0 \mid 1,5,0 \mid 2,1,3 \mid 4,3,y_5)$& $3 < y_5 < 5$ & 0,4\\
  5& $(4,3,y_5 \mid 5,4,5 \mid 2,1,0 \mid 2,1,0 \mid 3,2,3)$& $3 < y_5 < 5$ & 0,4\\
  6& $(5,4,y_5 \mid 5,4,5 \mid 2,1,0 \mid 2,1,0 \mid 3,2,4)$& $4 < y_5 < 5$ & 0\\
  7& $(4,3,y_5 \mid 1,5,1 \mid 2,1,0 \mid 2,1,0 \mid 3,2,3)$& $3 < y_5 < 1$ & 0,4,5\\
  8& $(5,4,y_5 \mid 1,5,1 \mid 2,1,0 \mid 2,1,0 \mid 3,2,4)$& $4 < y_5 < 1$ & 0,5\\
  9& $(4,3,4 \mid 5,4,y_5 \mid 1,5,1 \mid 3,2,0 \mid 3,2,0)$& $4 < y_5 < 1$ & 0,5\\
 10& $(4,3,5 \mid 1,5,y_5 \mid 1,5,1 \mid 3,2,0 \mid 3,2,0)$& $5 < y_5 < 1$ & 0\\
 11& $(4,3,4 \mid 5,4,y_5 \mid 2,1,2 \mid 3,2,0 \mid 3,2,0)$& $4 < y_5 < 2$ & 0,1,5\\
 12& $(4,3,5 \mid 1,5,y_5 \mid 2,1,2 \mid 3,2,0 \mid 3,2,0)$& $5 < y_5 < 2$ & 0,1\\
 13& $(4,3,0 \mid 5,4,5 \mid 1,5,y_5 \mid 2,1,2 \mid 4,3,0)$& $5 < y_5 < 2$ & 0,1\\
 14& $(4,3,0 \mid 5,4,1 \mid 2,1,y_5 \mid 2,1,2 \mid 4,3,0)$& $1 < y_5 < 2$ & 0\\
 15& $(4,3,0 \mid 5,4,5 \mid 1,5,y_5 \mid 3,2,3 \mid 4,3,0)$& $5 < y_5 < 3$ & 0,1,2\\
 16& $(4,3,0 \mid 5,4,1 \mid 2,1,y_5 \mid 3,2,3 \mid 4,3,0)$& $1 < y_5 < 3$ & 0,2\\
 17& $(5,4,0 \mid 5,4,0 \mid 1,5,1 \mid 2,1,y_5 \mid 3,2,3)$& $1 < y_5 < 3$ & 0,2\\
 18& $(5,4,0 \mid 5,4,0 \mid 1,5,2 \mid 3,2,y_5 \mid 3,2,3)$& $2 < y_5 < 3$ & 0\\
 19& $(5,4,0 \mid 5,4,0 \mid 1,5,1 \mid 2,1,y_5 \mid 4,3,4)$& $1 < y_5 < 4$ & 0,2,3\\
 20& $(5,4,0 \mid 5,4,0 \mid 1,5,2 \mid 3,2,y_5 \mid 4,3,4)$& $2 < y_5 < 4$ &0,3\\
 \hline
\end{tabular}\\\\

However, the cases $y_5=0$ in rows 3,7,11,15,19 cannot occur because that would imply that the 3-degree decorations 125, 123, 145, 234, 345 are  decorations of $K$, which are not by Lemma 1.27 in \cite{MRSdiscretehull}.
Therefore the number of exterior decorations that pattern (2) gives are at most $5+2\times 15=35$.
Pattern (1) gives only one new exterior decoration, which is
$(5,4,0 \mid 1,5,0 \mid 2,1,0 \mid 3,2,0 \mid 4,3,0).$
Since all these exterior decorations occur in $K_3$, the collared tiles are $35+1=36$.

By Lemma 1.27 in \cite{MRSdiscretehull} the possible collared edges are of the form shown in Figure \ref{f:edgescabelon}.
\begin{figure}
  \centering
  \includegraphics[scale=1]{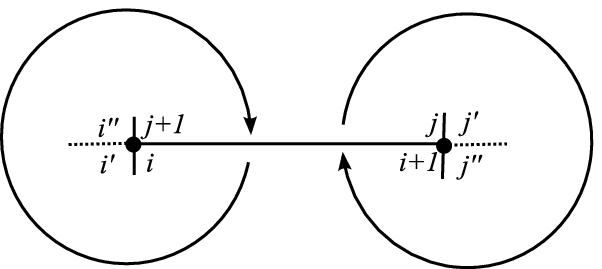}\\
  \caption{The decoration of a collared edge: $(i,i',i'',j+1\mid j, j',j'',i+1)$.  The dots mean that we might have 3- or 4- degree vertices.}\label{f:edgescabelon}
\end{figure}
\begin{figure}
  \centering
  \includegraphics[scale=.42]{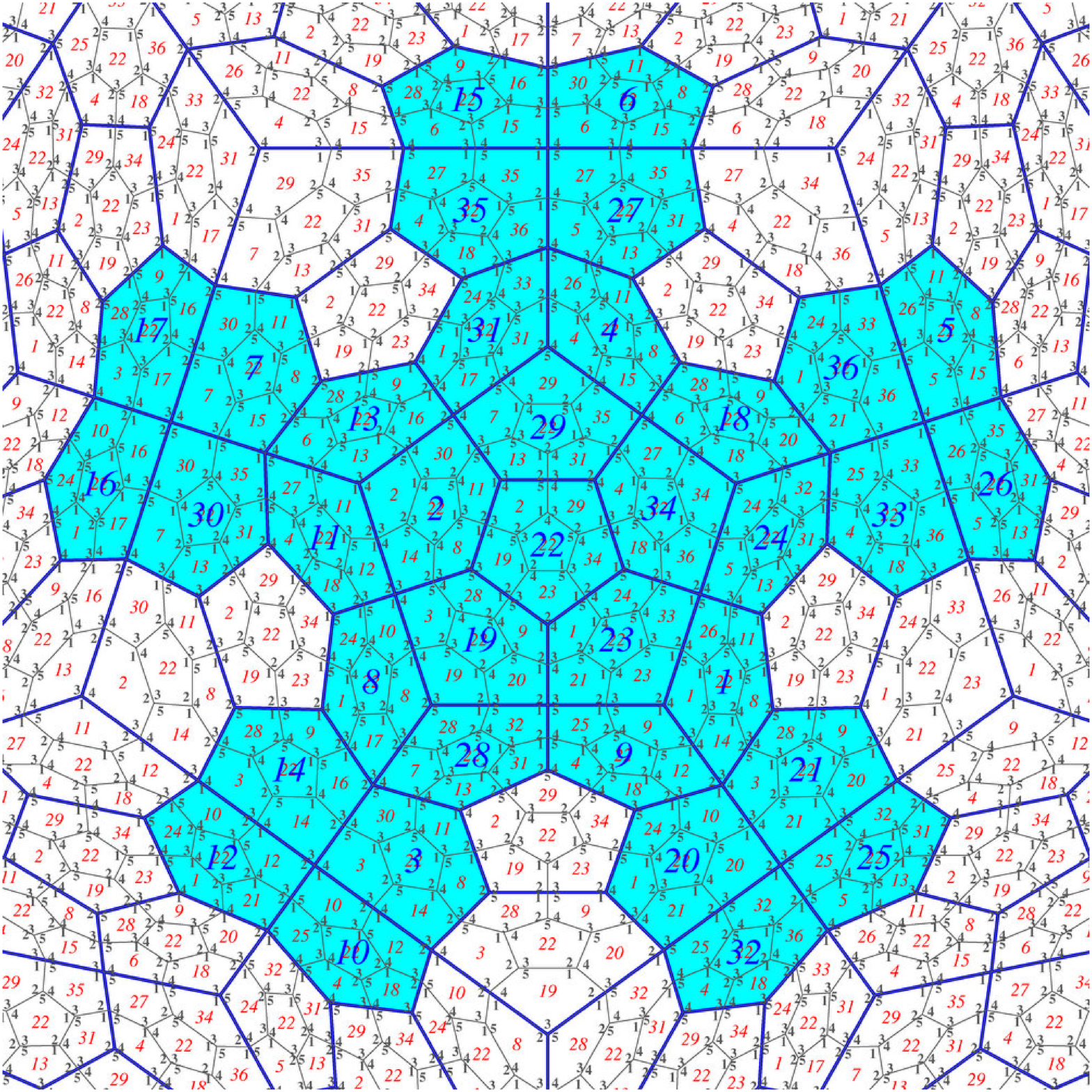}\\
  \caption{Collared substitution $\omega$ on a collared tile. }\label{f:collaredsubstitution}
\end{figure}
Reading the collared edges from our 36 collared pentagons, we get 45 collared edges, and this is of course under the observation that $(i,i',i'',j+1\mid j, j',j'',i+1)$ can also be written as $( j, j',j'',i+1 \mid i,i',i'',j+1)$, and if $i''=0$ then we can swap the values of $i'$ and $i''$.
It is interesting to notice that 18 collared edges start with $i$, where $i=1,2,3,4,5$.

Reading the collared vertices from our 45 collared edges, we get 10 collared vertices, which are precisely the ten decorated vertices given by Lemma 1.27 in \cite{MRSdiscretehull}.
\end{proof}
\begin{defn}[\index{collared substitution $\omega$ on a collared tile}collared substitution $\omega$ on a collared tile]\label{d:collaredsubsoncollaredtile}
Let\, $t_1,\cdots,t_{36}$ be the collared tiles of $K$ listed in Table \ref{t:collaredtiles}.
The collared substitution $\omega$ on a collared tile $t_i$ is defined as in the following list, which is read from Figure \ref{f:collaredsubstitution}.
\end{defn}

\begin{minipage}[b]{0.5\linewidth}
 $\omega(t_{1})= (t_{22},t_{11},t_{8},t_{14},t_{1},t_{26})$\\
$\omega(t_{2})= (t_{22},t_{11},t_{8},t_{14},t_{2},t_{30})$\\
$\omega(t_{3})= (t_{22},t_{11},t_{8},t_{14},t_{3},t_{30})$\\
$\omega(t_{4})= (t_{22},t_{11},t_{8},t_{15},t_{4},t_{26})$\\
$\omega(t_{5})= (t_{22},t_{11},t_{8},t_{15},t_{5},t_{26})$\\
$\omega(t_{6})= (t_{22},t_{11},t_{8},t_{15},t_{6},t_{30})$\\
$\omega(t_{7})= (t_{22},t_{11},t_{8},t_{15},t_{7},t_{30})$\\
$\omega(t_{8})= (t_{22},t_{10},t_{8},t_{17},t_{1},t_{24})$\\
$\omega(t_{9})= (t_{22},t_{9},t_{12},t_{18},t_{4},t_{25})$\\
$\omega(t_{10})=(t_{22},t_{10},t_{12},t_{18},t_{4},t_{25})$\\
$\omega(t_{11})=(t_{22},t_{11},t_{12},t_{18},t_{4},t_{27})$\\
$\omega(t_{12})=(t_{22},t_{10},t_{12},t_{21},t_{1},t_{24})$\\
$\omega(t_{13})=(t_{22},t_{9},t_{16},t_{13},t_{6},t_{28})$\\
$\omega(t_{14})=(t_{22},t_{9},t_{16},t_{14},t_{3},t_{28})$\\
$\omega(t_{15})=(t_{22},t_{9},t_{16},t_{15},t_{6},t_{28})$\\
$\omega(t_{16})=(t_{22},t_{10},t_{16},t_{17},t_{1},t_{24})$\\
$\omega(t_{17})=(t_{22},t_{9},t_{16},t_{17},t_{3},t_{28})$\\
$\omega(t_{18})=(t_{22},t_{9},t_{20},t_{18},t_{6},t_{28})$
  \end{minipage}
  \hspace{0.5cm}
  \begin{minipage}[b]{0.5\linewidth}
$\omega(t_{19})=(t_{22},t_{9},t_{20},t_{19},t_{3},t_{28})$\\
$\omega(t_{20})=(t_{22},t_{10},t_{20},t_{21},t_{1},t_{24})$\\
$\omega(t_{21})=(t_{22},t_{9},t_{20},t_{21},t_{3},t_{28})$\\
$\omega(t_{22})=(t_{22},t_{34},t_{23},t_{19},t_{2},t_{29})$\\
$\omega(t_{23})=(t_{22},t_{33},t_{23},t_{21},t_{1},t_{24})$\\
$\omega(t_{24})=(t_{22},t_{32},t_{31},t_{13},t_{5},t_{24})$\\
$\omega(t_{25})=(t_{22},t_{32},t_{31},t_{13},t_{5},t_{25})$\\
$\omega(t_{26})=(t_{22},t_{35},t_{31},t_{13},t_{5},t_{26})$\\
$\omega(t_{27})=(t_{22},t_{35},t_{31},t_{13},t_{5},t_{27})$\\
$\omega(t_{28})=(t_{22},t_{32},t_{31},t_{13},t_{7},t_{28})$\\
$\omega(t_{29})=(t_{22},t_{35},t_{31},t_{13},t_{7},t_{29})$\\
$\omega(t_{30})=(t_{22},t_{35},t_{31},t_{13},t_{7},t_{30})$\\
$\omega(t_{31})=(t_{22},t_{33},t_{31},t_{17},t_{1},t_{24})$\\
$\omega(t_{32})=(t_{22},t_{32},t_{36},t_{18},t_{4},t_{25})$\\
$\omega(t_{33})=(t_{22},t_{33},t_{36},t_{18},t_{4},t_{25})$\\
$\omega(t_{34})=(t_{22},t_{34},t_{36},t_{18},t_{4},t_{27})$\\
$\omega(t_{35})=(t_{22},t_{35},t_{36},t_{18},t_{4},t_{27})$\\
$\omega(t_{36})=(t_{22},t_{33},t_{36},t_{21},t_{1},t_{24})$.
  \end{minipage}
\begin{figure}
  \centering
  \includegraphics[scale=1]{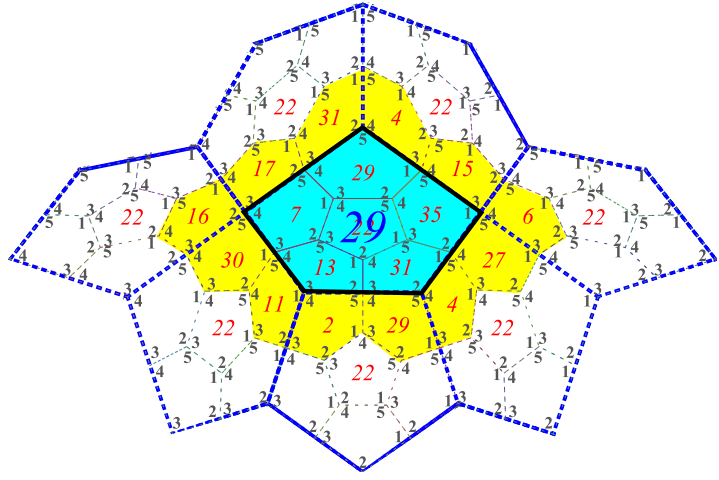}\\
  \caption{The collared substitution forces its border.}\label{f:forcingitsbordertile29}
\end{figure}
\begin{lem}
  The collared substitution $\omega$ forces its border for $k=1$.
\end{lem}
\begin{proof}
In Figure \ref{f:forcingitsbordertile29} we show $\omega(t_{29})$ in blue color, and its neighbors in yellow color, which are "forced by the substitution". The same applies for the other collared tiles.
\end{proof}
\newpage

\section{Anderson-Putnam finite CW-complex $\Gamma$}
In this section we will construct an equivalence relation on $\Omega$.
Informally, we will say that $\cisom{L,x}$ and $\cisom{L',x'}$ are equivalent if $x$ and $x'$ lie in the same collared tile in the same spot.
 This is unambiguous if the point $x$ is an interior point of the collared tile;
 if it lies on the boundary then we will use the level-k-supertile coming from the forcing the border condition and say that they are equivalent if $\omega^k(x)$ and $\omega^k(x')$ lie in the same level-k-supertile  in the same spot. We will now state this formally.

\begin{defn}[patch $T(x)$]
For a point $x\in L$ define \emph{$T(x)$} to be the patch consisting of all the collared tiles in $L$ containing $x$.
For a collared patch $p$ of $L$, define $T(p)$ to be the collared patch consisting of itself and of its neighboring collared tiles.
\end{defn}

\begin{defn}[$\sim_p$, $\sim_T$]\index{$\sim_p$}
We say that $\cisom{L,x}\sim_p \cisom{L',x'}$ if there are two collared tiles $x\in t\in L$ and $x'\in t'\in L'$ and an isometry $\psi:t\to t'$ that preserves the decorations from the first collared tile to the second one and $\psi(x)=x'$.
This binary relation $\sim_p$ is not transitive, so we take the transitive closure of it and we denote it with \emph{$\sim_T$}.
\end{defn}

\begin{lem}
  The equivalence relation $\sim_T$ is closed in $\Omega\times \Omega$.
\end{lem}
\begin{proof}
  Suppose that $\{(\cisom{L_n,x_n},\cisom{L'_n,x'_n})\}_{n\in\N}$ is a sequence in $\Omega\times \Omega$ converging to $(\cisom{L,x},\cisom{L',x'})\in \Omega\times \Omega$ such that
  $\cisom{L_n,x_n}\sim_T\cisom{L'_n,x'_n}$.
  Thus there exists a list
  \begin{scriptsize}
  $$\cisom{L_n,x_n}\sim_p \cisom{L^1_n,x^1_n}\sim_p \cisom{L^2_n,x^2_n}\sim_p\cdots \cisom{L^{m_n}_n,x^{m_n}_n}\sim_p \cisom{L'_n,x'_n}.$$
  \end{scriptsize}
  We can also assume that $m_n$ is at most the number of collared tiles because if it is greater than the number of collared tiles then we shorten the list by replacing the sublist that starts and ends on the same collared tile with just the collared tile.
  We can even assume that $m_n$ is a fixed integer $m$ by adding  $\cisom{L'_n,x'_n}\sim_p \cisom{L'_n,x'_n}$  at the end of each list $m-m_n$ times.
  Since $\Omega$ is compact, we can assume that for each $0\le k\le m+1$, $\cisom{L^k_n,x^k_n}$ converges to $\cisom{L^k,x^k}$ (by passing to a subsequence), where $L_n^0:=L_n$, $x_n^0:=x_n$, $L_n^{m+1}:=L'_n$, $x_n^{m+1}:=x_n'$.
  Thus, there exist $m+2$ sequences $x_n^k\in t_n^k\in  L^k_n$, where $0\le k\le m+1$.
  By passing to a subsequence $m+2$ times, we can assume that  for each $k$ all $t^k_n$ are the same, as we have a finite number of collared tiles.
  Since $d(\cisom{L^k_n,x^k_n},\cisom{L^k,x^k})\to 0$, for large $n$ the sequence $\phi^k_n(x^k_n)$ lives in the collared tile $t^k$ containing $x^k$ and converges to $x^k$, and $t^k$ is the same as $t^k_n$, as the maps $\phi^k_n$ coming from the distance definition are cell-preserving isometries.
  The isometry $\phi_n^{k+1}\circ\psi_n^k\circ(\phi^{k}_n)^{-1}:t^k\to t^{k+1}$ is independent of $n$ for large $n$, and so we define $\psi^k:=\phi_n^{k+1}\circ\psi_n^k\circ(\phi^{k}_n)^{-1}$. Since $\phi_n^k(x^k_n)$ converges to $x^k$ and $\psi^k$ is continuous, $\psi^k(x^k)=x^{k+1}$. Hence $\cisom{L^k,x^k}\sim_p\cisom{L^{k+1},x^{k+1}}$.
  Hence $\cisom{L,x}\sim_T\cisom{L',x'}$.
\end{proof}

\begin{lem}
  If $\cisom{L,x}\sim_p \cisom{L',x'}$ then there are two collared tiles $x\in t\in L$, $x'\in t'\in L'$ such that
  $T(\omega(t))\cong T(\omega(t'))$, where the isomorphism is cell-preserving, preserves the decorations of the collared tiles, is an isometry on each cell and maps $\omega(x)$ to $\omega(x')$.
\end{lem}
\begin{proof}
 By definition of $\sim_p$ there are two collared tiles $x\in t\in L$ and $x'\in t'\in L'$ and an isometry $\psi:t\to t'$ that preserves the decorations from the first collared tile to the second one and $\psi(x)=x'$.
 Since the collared subdivision $\omega$ forces its border with $k=1$, $\omega(t)$ and $\omega(t')$ agree on their collared neighbors. That is
 $T(\omega(t))\cong T(\omega(t'))$.
\end{proof}

\begin{cor}\label{c:classtildecotainsnbhtilesofx}
  If $\cisom{L,x}\sim_T \cisom{L',x'}$ then $T(\omega(x))\cong T(\omega(x'))$, with $\omega(x)$ mapped to $\omega(x')$. That is,
  $\omega(L,x)$ and $\omega(L',x')$ have in common all collared tiles containing $\omega(x)$.
\end{cor}

\begin{proof}
 If $\cisom{L,x}\sim_T \cisom{L',x'}$ then $$\cisom{L,x}\sim_p \cisom{L_1,x_1}\sim_p \cisom{L_2,x_2}\cdots\sim_p \cisom{L_m,x_m}\sim_p\cisom{L',x'}$$
 for some integer $m$.
  But $\cisom{L,x}\sim_p \cisom{L_1,x_1}$ implies that $T(\omega(t))\cong T(\omega(t_1)) $ for some collared tiles $x\in t\in L$, $x_1\in t_1\in L_1$ and the isomorphism identifies $x$ with $x_1$. Similarly,  $\cisom{L_1,x_1}\sim_p \cisom{L_2,x_2}$ implies that $T(\omega(t_1'))\cong T(\omega(t_2))$  for some collared tiles $x_1\in t_1'\in L_1$  and $x_2\in t_2\in L_2$.
   Since $T(\omega(x_1))\subset T(\omega(t_1)) \cap T(\omega(t_1'))$, we have
  $T(\omega(x))\cong T(\omega(x_2))$.
  By induction on $m$, we get the corollary.
\end{proof}


\begin{lem}\label{l:meaningofsimT}
Assume that $\cisom{L,x}\sim_T\cisom{\omega(L',x')}$.  We have
\begin{enumerate}
  \item If $x$ is an interior point of a collared tile then so is $\omega(x')$ and $x'$. In such case $\cisom{L,x}\sim_p \cisom{\omega(L',x')}$.
  \item If $x$ is an interior point of a collared edge or is a vertex then so is $\omega(x')$ but $x'$ might be an interior point of a tile.
\end{enumerate}
\end{lem}
\begin{proof}
 If $\cisom{L,x}\sim_T \cisom{L'',x''}$
 then $$\cisom{L,x}\sim_p \cisom{L_1,x_1}\sim_p \cisom{L_2,x_2}\cdots\sim_p \cisom{L_m,x_m}\sim_p\cisom{L'',x''}$$
 for some integer $m$.
  By definition of $\sim_p$ there are two collared tiles $x_i\in t_i\in L_i$ and $x_{i+1}\in \tilde t_{i+1}\in L_{i+1}$ and an isometry $\psi_i:t_i\to \tilde t_{i+1}$, where $i=0,\ldots, m$ and $x_0:=x$, $x_{m+1}:=x''$, that preserves the decorations from the first collared tile to the second one and $\psi_i(x_i)=x_{i+1}$. We consider the case when $x_0$ is in the interior of a tile.

   If $x_0$ is an interior point in $t_0$ then $\psi_0(x_0)=x_1$ is an interior point in $\psi_0(t_0)=\tilde t_1$. Since tiles do not overlap on their interiors, $\tilde t_1= t_1$ and so $x_1$ is an interior point in $t_1$. By finite induction, it follows that $\tilde t_{m+1}=t_{m+1}$ and that $x_{m+1}$ is an interior point in $t_{m+1}$. In summary, if $x$ is in the interior of a tile then $\cisom{L,x}\sim_T \cisom{L'',x''}$ implies $\cisom{L,x}\sim_p \cisom{L'',x''}$.
   Since $\omega$ on a tile is a homeomorphism, if $\omega(x_1)$ is in the interior then so is $x_1$. This proves part 1.
   Part 2 follows immediately from definition of $\omega$ on a tile. See Figure \ref{f:interiorboundarypoint}.
\end{proof}
\begin{figure}
  \centering
  \includegraphics[scale=1]{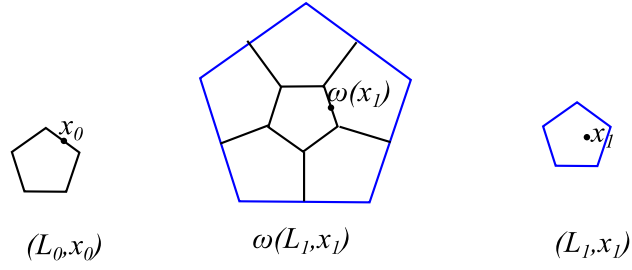}\\
  \caption{If $\cisom{L_0,x_0}\sim_T \cisom{\omega(L_1,x_1)}$ and if $x_0$ is in the boundary then so is $\omega(x_1)$ but $x_1$ might not be.}\label{f:interiorboundarypoint}
\end{figure}

\begin{lem}
  The quotient space $\Gamma:=\Omega/\sim_T$ is compact and Hausdorff.
  The map $\gamma:\Gamma\to\Gamma$  defined by$$\gamma([\cisom{L,x}]_{\sim_T}):=[\cisom{\omega(L,x)}]_{\sim_T}$$ is well-defined, continuous, and surjective.
\end{lem}
\begin{proof}
Since the space $\Omega$ is compact and Hausdorff and the equivalence relation $\sim_T$ is closed in $\Omega\times\Omega$, the quotient space $\Omega/\sim_T$ is Hausdorff.

Let $q:\Omega\to\Gamma$ be the quotient map.
Since $q$ is continuous and $\Omega$ is compact, $\Gamma$ is compact.

  The composition $q\circ \omega:\Omega\to\Gamma$ is continuous and surjective and it is given by $q\circ\omega(\cisom{L,x})=[\cisom{\omega(L,x)}]_{\sim_T}$.
  If $\cisom{L,x}\sim_T \cisom{L',x'}$ are equivalent, then by Corollary \ref{c:classtildecotainsnbhtilesofx} there is an isomorphism $\psi:T(\omega(x))\to T(\omega(x'))$ such that $\psi(\omega(x))=\omega(x')$ which  maps collared cells to collared cells and is isometric on each cell.
  Let $t\in T(\omega(x))$ be a collared tile containing $\omega(x)$. Then $\psi:t\to \psi(t)$ maps $\omega(x)$ to $\omega(x')$ and so
   $\cisom{\omega(L,x)}\sim_p\cisom{\omega(L',x')}$.
  We have thus shown that if $\cisom{L,x}\sim_T \cisom{L',x'}$ then $q\circ\omega(L,x)=q\circ\omega(L',x')$.
  Therefore $q\circ\omega$ descends to the quotient i.e. there exists a unique continuous map $\gamma:\Gamma\to\Gamma$ such that $q\circ\omega=\gamma\circ q$. Since the left hand side is surjective, $\gamma$ is surjective.
\end{proof}

\begin{pro}
The space $\Gamma$ has the structure of a finite CW-complex whose closed-2-cells are the 36 collared faces,
the closed-1-cells are the 45 collared edges, and the 0-cells are the 10 collared vertices.
\end{pro}
\begin{proof}
  Let $\Gamma'$ be the CW-complex constructed as follows. Start with the 10 collared vertices.
  Since each collared edge consists of two collared vertices, we can join the 10 collared vertices according to the collared edges.
  Similarly, we join the collared edges according to the collared faces. Thus, by construction $\Gamma'$ is a finite CW-complex and so it is compact and Hausdorff.
  Locally, $\Gamma'$ looks like this. Consider the special collared tile $t_{22}$. All the neighbors of $t_{22}$ are $\omega(t_i)$, $i=1,\ldots 36$, and recall that $\omega(t_i)$ is a half-dodecahedron, with $t_{22}$ at its center. The $CW$-complex $\Gamma'$ at $t_{22}$ looks like the union $\union_{i=1}^{36} \omega(t_i)$, where all the centers $t_{22}$ are identified, and also the corresponding collared faces edges and vertices.
  For another collared tile $t_i$, $i\ne 22$, we need to find all the neighbors of $t_i$, and we call them $\mathrm{nbh}_j(t_i)$, $j=1,\ldots, n_i$ for some integer $n_i\in\N$.
  We know that there are finitely many, i.e. $n_i\in\N$, for $K$ satisfies the finite local complexity (FLC) by Theorem 1.14 in \cite{MRSdiscretehull}.  The $CW$-complex $\Gamma'$ at $t_{i}$ looks like the union $\union_{j=1}^{n_i} \mathrm{nbn}_j(t_i)$, where all the  $t_{i}$'s are identified, and also the corresponding collared faces edges and vertices.

  Let $\Gamma\to\Gamma'$ be given by $[\cisom{L,x}]_{\sim_T}\mapsto x$.
  We now show this map is well defined. Suppose that $x\in t^\circ$ is in the interior of a collared tile $t\in L$.
  By the proof of Lemma \ref{l:meaningofsimT},  $\cisom{L,x}\sim_T\cisom{L',x'}$ implies $\cisom{L,x}\sim_p\cisom{L',x'}$ and so $x'$ is contained in a copy of the same collared tile $t$. Hence the map is independent of the representative when $x$ is in the interior of a collared tile of $L$.
   Similarly, if $x\in e^\circ$ is in the interior of a collared edge $e\in L$, then  $\cisom{L,x}\sim_T\cisom{L',x'}$ implies that $x'$  is contained in a copy of the same collared edge $e$. Hence the map is independent of the representative when $x$ is in the interior of a collared edge of $L$.
   The same result holds if $x$ is a collared vertex. By construction this map is surjective and if $x=x'$ then $\cisom{L,x}\sim_T\cisom{L',x'}$, hence injective.
   The map is also continuous for if $x\in t^\circ$ is in the interior of a collared tile $t\in L$, there is a small ball $B_r(x)\subset t^\circ$ and
$B_{r'}(\cisom{L,x'})\subset \{[\cisom{L,x'}]_{\sim_T}\subset \Omega\mid x'\in B_r(x)\subset t^\circ\}$ for some small $r'$ such that $B_{r'}(x')\subset  B_r(x)$, hence open (and we do the same for when $x$ is in an edge or a vertex). We remark that  $Z(B(x,n,L),x)\subset [\cisom{(L,x)}]_{\sim_T}\subset\Omega$ for any $n>2$, where $Z(B(x,n,L),x)\subset \Omega$ is the set of all the tilings containing the ball $(B(x,n,L),x)$, and the equivalence class $[\cisom{(L,x)}]_{\sim_T}$ is seen as a subset of $\Omega$.
Since $\Gamma$ is compact and $\Gamma'$ is Hausdorff, $\Gamma'$ and $\Gamma$ are homeomorphic.
\end{proof}

\begin{rem}
Each collared edge joins either 2,4,7 or 8 collared tiles. See Figure \ref{f:collaredfacesintermsofedges}.
Twenty-five collared tiles have 4 distinct collared vertices (one of them repeats).
So eleven collared tiles have 5 distinct collared vertices. See Figure \ref{f:cverticesoneachctile}.
Each 3-degree collared vertex joins 7 collared edges and 17 collared tiles.
Each 4-degree collared vertex  joins 11 collared edges and 14 collared tiles.
\end{rem}

\begin{figure}[htbp]
  \begin{minipage}[b]{0.5\linewidth}
    \centering
    \includegraphics[width=\linewidth]{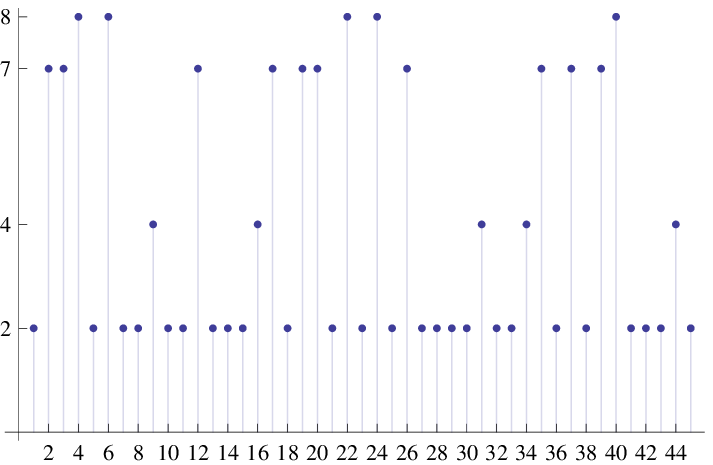}
    \caption{Number of collared tiles sharing same collared edge.}
    \label{f:collaredfacesintermsofedges}
  \end{minipage}
  \hspace{0.5cm}
  \begin{minipage}[b]{0.5\linewidth}
    \centering
    \includegraphics[width=\linewidth]{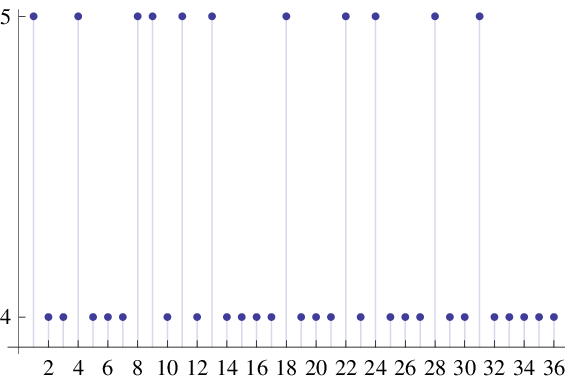}
    \caption{Number of distinct collared vertices on each collared tile.}
    \label{f:cverticesoneachctile}
  \end{minipage}
\end{figure}

\section{Inverse limits}
We define the inverse limit space from the inverse system
$$\xymatrix{\Gamma&\Gamma\ar[l]^{\gamma}&\Gamma\ar[l]^{\gamma}&\Gamma\ar[l]^{\gamma}&\ar[l]^{\gamma}\cdots}$$
as $$\Omega_1:=\{(x_0,x_1,\cdots)\in \prod\Gamma\mid x_i=\gamma(x_{i+1}), i\in\N_0\}.$$
 Observe that if we know $x_n$ then automatically we know $x_{n-1},\ldots,x_0$.
The correct notation for $\Omega_1$ should be $\lim_\leftarrow(\Gamma,\gamma)$. However, we find more convenient to write $\Omega_1$ to resemble $\Omega$ and to have a more straightforward notation.
 We equip $\prod \Gamma$ with the product topology, and $\Omega_1\subset \prod \Gamma$ with the subspace topology.
 A basis for the subspace topology is the collection of sets of the form
 \begin{eqnarray*}
  B(U,n)&:=&\pi_n^{-1}(U)\intersection \Omega_1\\
  &=&\{(\gamma^n(x_n),\ldots,\gamma(x_n),x_n,x_{n+1},\ldots)\in \Omega_1\mid x_n\in U\}\\
  &=&\{(x_0,x_1,\ldots)\in \Omega_1\mid x_i\in \gamma^{n-i}(U), i=0,\ldots,n\},
 \end{eqnarray*}
 where $U$ is an open subset of $\Gamma$,  $n\in\N_0$, and $\pi_n: \prod \Gamma\to \Gamma$ is the $n$-th projection of the product space $\prod \Gamma$ into $\Gamma$, which in the product topology is by definition continuous.

\begin{lem}
  The inverse limit space $\Omega_1$ is Hausdorff, compact, and closed in $\prod\Gamma$.
\end{lem}
\begin{proof}
 By Tychonoff's theorem, the infinite product $\prod\Gamma$ with the product topology is compact, as $\Gamma$ is compact.
 It is also Hausdorff because $\Gamma$ is Hausdorff. Hence $\Omega_1$ is Hausdorff.
 The inverse limit space $\Omega_1$ is a closed subset of $\prod\Gamma$. Indeed, suppose $\{\vx^{(n)}\}_{n\in\N}$ is a sequence in $\Omega_1$ and suppose that $\vx^{(n)}\to \vx\in\prod \Gamma$. We need to show that  $\gamma(x_{i+1})=x_i$ for all $i\in \N_0$, where $\vx=(x_0,x_1,\ldots)$. In the product topology, also known as the topology of pointwise convergence, $\vx^{(n)}\to \vx$  means $x^{(n)}_i\to x_i$ for all $i\in\N_0$. Since $\gamma$ is continuous, $\gamma(x^{(n)}_{i+1})\to \gamma(x_{i+1})$. Since  $\gamma(x^{(n)}_{i+1})=x_i^{(n)}\to x_i$ and $\Pi\Gamma$ is Hausdorff, the limit is unique and so  $\gamma(x_{i+1})=x_i$.
 Thus $\Omega_1$ is closed, hence compact.
\end{proof}

The map $\gamma$ induces a right shift map defined in the following lemma.
\begin{lem}
  The right shift map $\rho:\Omega_1 \to \Omega_1$ and left shift map $\rho^{-1}:\Omega_1 \to \Omega_1$ defined by
  \begin{eqnarray*}
    &&\rho(\vx):=(\gamma(x_0),\gamma(x_1),\gamma(x_2),\ldots)=(\gamma(x_0),x_0,x_1,\ldots)\\
    &&\rho^{-1}(\vy)=(y_1,y_2,\ldots)
  \end{eqnarray*}
 are continuous and inverse of each other where $\vx=(x_0,x_1,\ldots)$ and $\vy=(y_0,y_1,\ldots)$.
\end{lem}
\begin{proof}
For $n\in\N_0$ we have
\begin{eqnarray*}
  \rho^{-1}(\pi_n^{-1}(U))&=&\{\vx\in\Omega_1\mid \rho(\vx)\in\pi^{-1}_n(U)\}\\
  &=&\{\vx\in\Omega_1\mid \pi_n(\rho(\vx))\in U\}\\
  &=&\{\vx\in\Omega_1\mid \pi_n(\gamma(x_0),\gamma(x_1),\ldots)\in U\}\\
  &=&\{\vx\in\Omega_1\mid \gamma(x_n)\in U\}\\
  &=&\{\vx\in\Omega_1\mid x_n\in \gamma^{-1}(U)\}\\
  &=&\pi_n^{-1}(\gamma^{-1}(U)).
\end{eqnarray*}
It follows that $\rho$ is continuous as $\gamma$ is continuous. Since $\Omega_1$ is compact and Hausdorff and $\rho$ is continuous and bijective, it is a homeomorphism.
\end{proof}
Notice that $\rho(\vx)$ is simply $\vx$ with the entry $\gamma(x_0)$ appended at the beginning. Also $\rho^{-1}(\vy)$ is $\vy$ with the first entry $y_0$ removed. Interesting enough,  $\rho^{-1}$ removes the first entry, but $\gamma$ remembers it and so we can recover it.

Since any homeomorphism induces a $\Z$-action, and $\rho$ is a homeomorphism, we have a dynamical system $(\Omega_1, \rho)$.

The sequence $\{[\omega^{-n}(\cisom{L,x})]_{\sim_T}\}_{n\in\N_0}$ is obviously in the inverse limit $\Omega_1$.
The following theorem shows that all the elements in the inverse limit $\Omega_1$ are actually of this form.

\begin{thm}\label{t:topologicalconjugateOmegaOmega1}
  The dynamical systems $(\Omega,\omega)$ and $(\Omega_1,\rho)$ are topological conjugate.
\end{thm}
\begin{proof}
  Define $\pi:\Omega\to \Omega_1$ by
  $$\pi(\cisom{L,z})=([\cisom{L,z}]_{\sim_T},[\omega^{-1}(\cisom{L,z})]_{\sim_T},[\omega^{-2}(\cisom{L,z})]_{\sim_T},\cdots).$$
  We start by showing that the map is injective. Suppose that $\pi(\cisom{L,z})=\pi(\cisom{L',z'})$.
  Then $[\omega^{-n}(\cisom{L,z})]_{\sim_T}=[\omega^{-n}(\cisom{L',z'})]_{\sim_T}$ for any $n\in\N_0$.
  By Corollary \ref{c:classtildecotainsnbhtilesofx} we have  that $T(\omega^{-n+1}(z))\cong T(\omega^{-n+1}(z'))$, with the homeomorphism being cell-preserving and an isometry on each cell. Let $r$  be the smallest distance from $\omega^{-n+1}(z)$ to the boundary of $T(\omega^{-n+1}(z))$.
   Since the quotient map $q$ is well defined, we get $\omega^{n-1}(T(\omega^{-n+1}(z)))\cong \omega^{n-1}(T(w^{-n+1}(z')))$ and thus $\cisom{L,z}$ and $\cisom{L',z'}$ agree on a ball of radius $r \lambda^{n-1}$, $\lambda=1/0.54$ by Lemma 3.8 in \cite{MRScontinuoushull}.
   Since this holds for any $n$ and $\lambda>1$, we have  $d(\cisom{L,z},\cisom{L',z'})=0$ and so $\cisom{L,z}=\cisom{L',z'}$.

   We will now show it is surjective.
   Let $\{[\cisom{L_n,z_n}]_{\sim_T}\}_{n\in\N_0}$ be an element in the inverse limit $\Omega_1$.
   Thus $$[\cisom{L_{n-1},z_{n-1}}]_{\sim_T}=\gamma([\cisom{L_n,z_n}]_{\sim_T})=[\omega(\cisom{L_n,z_n})]_{\sim_T}.$$
   By Corollary \ref{c:classtildecotainsnbhtilesofx} we have  $T(\omega(z_{n-1}))\cong T(\omega^2(z_n))$.
   Thus
   $$T(\omega(z_1))\cong T(\omega^2(z_2))\subset \omega(T(\omega(z_2)))\cong \omega(T(\omega^2(z_3)))\subset \omega^2(T(\omega(z_3))) \cdots.$$
   Let $$L:=\lim_{n\to\infty} \omega^{n-1}(T(\omega(z_n))), \qquad z:=\lim_{n\to\infty} \omega^n(z_n)$$
   be the direct limit. Notice that $\omega(z_1)$ is being mapped to $\omega^2(z_2)$, which in turn is mapped to $\omega^3(z_3)$, and so on.
   Since for $n\in\N$, $T(\omega(z_n))\subset \omega^{-(n-1)}(L)$, we have $$[\cisom{\omega^{-(n-1)}(L,z)}]_{\sim_T}=[\omega(\cisom{L_n,z_n})]_{\sim_T}=[\cisom{(L_{n-1},z_{n-1})}]_{\sim_T}.$$
   Hence, $\pi(\cisom{L,z})=\{[\cisom{L_n,z_n}]_{\sim_T}\}_{n\in\N_0}$, so $\pi$ is surjective.
   An alternative method to show that $\pi$ is surjective is the following, where we use compactness. Let $\{[\cisom{L_n,z_n}]_{\sim_T}\}_{n\in\N_0}$ be an element in the inverse limit $\Omega_1$. Notice that $[\cisom{L_n,z_n}]_{\sim_T}=[\cisom{\omega^{k-n}(L_{k},z_{k})}]_{\sim_T}$ for any $k,n\in\N_0$ with $k\ge n$.
   Thus informally $\pi([\cisom{\omega^{\infty}(L_\infty,z_\infty)}]_{\sim_T})=\{[\cisom{L_n,z_n}]_{\sim_T}\}_{n\in\N_0}$, so the question is to specify $\omega^{\infty}(L_\infty,z_\infty)$ in terms of the sequence. Since $\{\cisom{L_n,z_n}\}_{n\in\N_0}$ is a sequence in $\Omega$ and $\Omega$ is compact, there is a convergent subsequence $\{\cisom{L_{n_i},z_{n_i}}\}_{i\in\N}$ converging to some $\cisom{L,z}\in\Omega$.
    So we specify $\omega^{\infty}(L_\infty,z_\infty)$ as $(L,z)$.


   We will now show that $\pi$ is continuous.
   Let $B(U,n)\subset\Omega_1$ be a basis element of $\Omega_1$.  That is, $B(U,n)=\pi_n^{-1}(U)\cap \Omega_1$, $U\subset \Gamma$ open, $n\in\N_0$. We have
   \begin{eqnarray*}
     \pi^{-1}(B(U,n))&=&\{\cisom{L,z}\in\Omega\mid \pi([\cisom{L,z}]_{\sim_T})\in \pi_{n}^{-1}(U)\cap \Omega_1\}\\
&=&\{\cisom{L,z}\in\Omega\mid \pi_n(\pi([\cisom{L,z}]_{\sim_T}))\in  U\}\\
     &=&\{\cisom{L,z}\in\Omega\mid [\cisom{\omega^{-n}(L,z)}]_{\sim_T}\in U\}\\
          &=&\{\cisom{L,z}\in\Omega\mid \cisom{\omega^{-n}(L,z)}\in q^{-1}(U)\}\\
          &=&\{\cisom{L,z}\in\Omega\mid \cisom{L,z}\in \omega^{n}(q^{-1}(U))\}\\
          &=& \omega^{n}(q^{-1}(U)),
   \end{eqnarray*}
   where $q:\Omega\to\Omega/\sim_T$ is the quotient map. Since $\omega^{-1}$ and $q$ are continuous maps and  $U$ is open, $\pi^{-1}(B(U,n))$ is open. Hence $\pi$ is continuous. Since $\Omega$ is compact and $\Omega_1$ is Hausdorff and $\pi:\Omega\to\Omega_1$ is a continuous bijection, $\pi$ is a homeomorphism.
   The substitution map $\omega$ is topological conjugate to the right shift map $\rho$ because
   \begin{eqnarray*}
     \pi\circ\omega(\cisom{L,z})&=&\pi(\cisom{\omega(L,z)})\\
     &=&([\cisom{\omega^{-i}(\omega(L,z))}]_{\sim_T})_{i\in\N_0}\\
     &=&(\gamma([\cisom{\omega^{-i}(L,z)}]_{\sim_T}))_{i\in\N_0}\\
     &=&\rho(\pi(\cisom{L,z})).
   \end{eqnarray*}
\end{proof}

The above theorem enable us to compute the cohomology of the hull $\Omega$.

\subsection*{Acknowledgments.}
The results of this paper were obtained during my Ph.D. studies at University of Copenhagen. I would like to express deep gratitude to my supervisor Erik Christensen and Ian F. Putnam whose guidance and support were crucial for the successful completion of this project.


\begin{thebibliography}{99}

\bibitem{RiemSurfBeardonBook}
  A. F. Beardon,
  \emph{ A Primer on Riemann Surfaces.}
  Cambridge University Press,
  1984.

\bibitem{Bellissard06}
  Jean Bellissard, Riccardo Benedetti, and Jean-Marc Gambaudo,
  \emph{Spaces of Tilings, Finite Telescopic Approximations
and Gap-Labeling.}
  Commun. Math. Phys. 261,
  (2006) 1-41.

\bibitem{buzziacim}
  Jerome Buzzi,
  \emph{A.C.I.M.'s For Arbitrary Expanding Piecewise $\R$-Analytic Mappings Of The Plane.}\\
    Ergod. Th. and Dynam. Sys,
    1999.



\bibitem{AConnes}
  A. Connes,
  \emph{ Non-commutative Geometry.}\\
 Academic Press,
 San Diego  (1994).

\bibitem{FloydFiniteSubdivisionRules01}
   J. W. Cannon, W. J. Floyd, and W. R. Parry,
   \emph{Finite subdivision rules.}
   http://www.math.vt.edu/people/floyd/research/papers/fsr.pdf,
   (2001).





\bibitem{Hatcher02}
  Allen Hatcher,
  \emph{ Algebraic Topology.}
  Cambridge University Press,
  2002.

\bibitem{Hilion11}
   Nicolas Bedaride, Arnaud Hilion,
   \emph{Geometric realizations of 2-dimensional substitutive tilings.}
   	arXiv:1101.3905 [math.GT],
   (2011).

\bibitem{TopologyJanich84}
   K. J\"anich, S. Levy,
   \emph{Topology.}
   Springer-Verlag New York Inc.,
   (1984).



\bibitem{Kellendonk95}
 Johannes Kellendonk,
  \emph{The Local Structure of Tilings and their
Integer Group of Coinvariants.}
Communications in Mathematical Physics,
(1997) 115-157.



\bibitem{May99}
  J. P. May,
  \emph{ A concise course in Algebraic Topology.}
  The Univeristy of Chicago Press.
  Chicago and London,
  1999.


\bibitem{Mozes89}
   Shahar Mozes,
   \emph{Tilings, substitution systems and dynamical systems generated by them.}
   Journal D'analyse Mathematique,
   (1989).


\bibitem{mrw87}
  Paul S. Muhly, Jean N. Renault, Dana P. Williams
  \emph{Equivalence and isomorphism for groupoid $C^*$-algebras.}\\
   J. Operator Theory,
   (1987) 3-22.


\bibitem{Putnam00OrderedKtheory}
 Ian F. Putnam,
 \emph{The ordered $K$-theory of $C^*$-algebras associated with substitution tilings.}
 Commun. Math. Phys. 214,
 (2000) 593-605.

\bibitem{PutnamBible95}
  Jared E. Anderson and Ian F. Putnam.
  \emph{Topological Invariants for Substitution Tilings and their Associated $C^*$-algebras.}
  Department of Mathematics and Statistics, University of Victoria, Victoria  B.C. Canada.
  (1995) 1-45.



\bibitem{PutnamCstarKtheory00}
  Johannes Kellendonk and Ian F. Putnam,
  \emph{Tilings, $C^*$-algebras and $K$-theory.}
  Directions in mathematical quasicrystals, CRM Monogr. Ser., 13, Amer. Math. Soc., Provicence, RI
  (2000) 177-206.



\bibitem{Putnametalenotes}
   Ian F. Putnam,
  \emph{ Orbit equivalence of Cantor minimal systems:Kyoto Winter School 2011.}\\
    http://www.math.uvic.ca/faculty/putnam/r/Kyoto\_2011\_main.pdf\\
  2011.


\bibitem{putnametalenotesreal}
  Jason Peebles, Ian F. Putnam, Ian Zwiers
  \emph{Minimal Dynamical Systems on the Cantor Set.}\\
   Lecture notes,
   (2011).


   \bibitem{MRSnonFLCpentTiling}
   Maria Ramirez-Solano,
  \emph{ A non FLC regular pentagonal tiling of the plane.}\\
   arXiv:1303.2000,
  2013.

\bibitem{MRSdiscretehull}
   Maria Ramirez-Solano,
  \emph{ Construction of the discrete hull for the combinatorics of a  regular pentagonal tiling of the plane.}
   arXiv:1303.5375,
  2013.

\bibitem{MRScontinuoushull}
   Maria Ramirez-Solano,
  \emph{  Construction of the continuous hull for the combinatorics of a regular pentagonal tiling of the plane.}
   arXiv:1303.5676,
  2013.




\bibitem{Sadun08}
  Lorenzo Sadun,
  \emph{ Topology of Tiling Spaces.}
  University Lecture Series Vol. 46,
  Providence, Rhode Island,
  2008.

\bibitem{SadunTilingSpacesAreCSFB}
   Lorenzo Sadun, R. F. Williams,
  \emph{ Tiling Spaces Are Cantor Set Fiber Bundles.}\\
   http://arxiv.org/pdf/math/0105125.pdf\\
  2001.



\bibitem{StephensonBowers97}
  Philip L. Bowers and Kenneth Stephenson,
  \emph{A "regular" pentagonal tiling of the plane.}
  Conformal geometry and dynamics.
  An electronic journal of the American Mathematical Society,
  (1997) 58-86.


\end{thebibliography}
\end{document}